\documentclass{article}
\usepackage{arxiv}
\usepackage[utf8]{inputenc} 
\usepackage[T2A]{fontenc}    
\usepackage[english,russian]{babel}
\usepackage{mmap}
\usepackage{hyperref}       
\usepackage{url}            
\usepackage{booktabs}       
\usepackage{amsfonts}       
\usepackage{nicefrac}       
\usepackage{amsmath,amssymb,amsthm}
\usepackage{mathtools}
\usepackage{algorithm, algorithmic}
\floatname{algorithm}{Алгоритм}
\usepackage{graphicx}
\usepackage{stackengine}    
\usepackage[labelsep=period]{caption}
\usepackage{hyperref}
\hypersetup{
  colorlinks   = true,    
  urlcolor     = green,   
  linkcolor    = blue,    
  citecolor    = red      
}
\usepackage{comment}
\usepackage{multirow}
\makeatletter
\usepackage{svg}
\usepackage{titlesec}
\titlelabel{\thetitle. }

\def\keywordname{{\bfseries Ключевые слова:}}%
\def\keywords#1{\par\addvspace\medskipamount{\rightskip=0pt plus1cm
\def\and{\ifhmode\unskip\nobreak\fi\ $\cdot$ }\noindent\keywordname\enspace\ignorespaces#1\par}}
\DeclareMathOperator*{\argmin}{arg\,min}

\theoremstyle{plain}
\newtheorem{theorem}{Теорема}

\newtheorem{definition}{Определение}

\theoremstyle{definition}
\newtheorem{remark}{Замечание}
\newtheorem{example}{Пример}

\renewcommand{\phi}{\varphi}
\renewcommand{\epsilon}{\varepsilon}
\renewcommand{\leq}{\leqslant}
\renewcommand{\geq}{\geqslant}

\newcommand{\He}{\section}

\author{
{\bfseries С.~С.~Аблаев\,$^{\it{a},\,\it{b},\,*}$,  Ф.~С.~Стонякин\,$^{\it{a},\,\it{b},\,**}$, М~С.~Алкуса\,$^{\it{a},\,\it{c},\,***}$, Д.~А.~Пасечнюк\,$^{\it{a},\,\it{d},\,\it{****}}$}
\\ {\itshape $^{\it{a}}$\,Московский физико-технический институт,}
\\ {\slshape 141701,} {\itshape г. Долгопрудный, Институтский пер.,~9}
\\ {\itshape $^{\it{b}}$\,Крымский федеральный университет им. В.\,И.\,Вернадского,}
\\ {\slshape 295007,} {\itshape г.~Симферополь, проспект академика Вернадского,~4}
\\ {\itshape $^{\it{c}}$\,Национальный исследовательский университет <<Высшая школа экономики>>,}
\\ {\slshape 101000,} {\itshape г. Москва, ул. Мясницкая, д.~20}
\\ {\itshape $^{\it{d}}$\,Исследовательский центр доверенного искусственного интеллекта ИСП РАН,}
\\ {\slshape 109004,} {\itshape г.~Москва, ул. Александра Солженицына,~25}
\\ {\itshape *\,e-mail: seydamet.ablaev@yandex.ru, ORCID: 0000-0002-9927-6503}
\\ {\itshape **\,e-mail: fedyor@mail.ru, ORCID: 0000-0002-9250-4438}
\\ {\itshape ***\,e-mail: mohammad.alkousa@phystech.edu, ORCID: 0000-0001-5470-0182}
\\ {\itshape ****\,e-mail: dmivilensky1@gmail.com, ORCID: 0000-0002-1208-1659}}

\title{АДАПТИВНЫЕ МЕТОДЫ ДЛЯ ВАРИАЦИОННЫХ НЕРАВЕНСТВ С ОТНОСИТЕЛЬНО ГЛАДКИМИ И ОТНОСИТЕЛЬНО СИЛЬНО МОНОТОННЫМИ ОПЕРАТОРАМИ}

\renewcommand{\shorttitle}{АДАПТИВНЫЕ МЕТОДЫ ДЛЯ ВАРИАЦИОННЫХ НЕРАВЕНСТВ С ОТНОСИТЕЛЬНО ГЛАДКИМИ И ОТНОСИТЕЛЬНО СИЛЬНО МОНОТОННЫМИ ОПЕРАТОРАМИ}

\date{}

\begin{document}

\maketitle

\begin{abstract}
Статья посвящена некоторым адаптивным методам для вариационных неравенств с относительно гладкими и относительно сильно монотонными операторами. Отталкиваясь от недавно предложенного проксимального варианта экстраградиентного метода для такого класса задач, мы детально исследуем метод с адаптивно подбираемыми значениями параметров. Доказана оценка скорости сходимости этого метода. Результат обобщён на класс вариационных неравенств с относительно сильно монотонными $\delta$-обобщённо гладкими операторами вариационного неравенства. Для задачи гребневой регрессии и вариационного неравенства, связанного с параллелепипедно-симплексными играми, выполнены численные эксперименты, демонстрирующие эффективность предложенной методики адаптивного подбора параметров в ходе реализации алгоритма.
\keywords{вариационное неравенство \and относительно сильно монотонный оператор \and относительно гладкий оператор \and адаптивный метод \and минимизация эмпирического риска \and  параллелепипедно-симплексная игра}
\end{abstract}

\He{ВВЕДЕНИЕ}

Свойства относительной гладкости и относительной сильной монотонности для вариационных неравенств были введены довольно недавно в \cite{Stonyakin} и \cite{Sid} соответственно как аналоги свойств $L$-гладкости и относительной сильной выпуклости для минимизационных задач. В данной работе, отправляясь от неадаптивного варианта проксимального зеркального метода из \cite{Sid}, для относительно гладких относительно сильно монотонных операторов детально исследован адаптивный алгоритм такого типа, анонсированный ранее в трудах конференции <<OPTIMA-2022>> \cite{Titov}. Более того, по аналогии с \cite{Stonyakin} рассмотрены и аналоги такого подхода на классе вариационных неравенств с относительно сильно монотонными и  $\delta$-обобщённо гладкими операторами, получены оценки качества выдаваемого решения, описывающие влияние на него параметра $\delta$.

Будем рассматривать задачу нахождения решения $x_*$ (также называемого слабым решением) вариационного неравенства
\begin{equation}\label{eq:1}
\max_{x \in Q} \langle g(x), x_* - x \rangle \leq 0,
\end{equation}
где $Q$~--- выпуклое замкнутое подмножество $\mathbb{R}^n$,
$g: Q \longrightarrow \mathbb{R}^n$. Предположим, что удовлетворяющее \eqref{eq:1} решение $x_*$ существует.

Всюду далее будем предполагать, что нам доступна некоторая выпуклая (вообще говоря, не сильно выпуклая) дифференцируемая прокс-функция $d$, порождающая расстояние, а также соответствующая ей дивергенция (расхождение) Брэгмана
\begin{equation}\label{Brg_form}
V(y, x) := d(y) - d(x) - \langle \nabla d(x), y - x \rangle.
\end{equation}
В работе \cite{Bauschke} выделен класс гладких задач относительно произвольной дивергенции Брэгмана (порождённой необязательно сильно выпуклыми $d$) и рассмотрен ряд примеров.
Напомним следующий аналог понятия относительной сильной выпуклости функции \cite{Lu_Nesterov_2018} для вариационных неравенств \cite{Sid}.
\begin{definition}\label{DefRelStrongMonot}
Оператор $g$ называется относительно $\mu$-сильно монотонным, где $\mu >0$, если для всяких $x, y \in Q$ верно неравенство
	\begin{equation}\label{eq:3}
	 	\mu V(y, x) + \mu V(x, y) \leq \langle g(y) - g(x), y - x \rangle.
	 \end{equation}
\end{definition}
Поясним на следующем примере, почему относительная сильная монотонность оператора вводится имен\-но согласно  \eqref{eq:3}.
\begin{example}
	Если $f: Q \longrightarrow \mathbb{R}$~--- относительно $\mu$-сильно выпуклая дифференцируемая функция
	\begin{equation}\label{eqrelativestorngconv}
	f(x) - f(y) + \mu V(x, y) \leq \langle \nabla{f(x)}, x - y \rangle \quad   \forall x, y \in Q,
	\end{equation}
	то
	\begin{equation}
	f(y) - f(x) + \mu V(y, x) \leq \langle \nabla{f(y)}, y - x \rangle \quad   \forall x, y \in Q.
	\end{equation}
	После сложения двух последних неравенств получаем $\forall x, y \in Q$
	\begin{align*}
	\mu V(x, y) + \mu V(y, x)\leq \langle \nabla{f(y)} - \nabla{f(x)}, y - x \rangle.
	\end{align*}
	Таким образом, неравенство \eqref{eq:3} верно при $g(x) = \nabla{f(x)}$, где $\nabla{f(x)}$~--- градиент $f$ в точке $x$.
\end{example}

Относительно сильно выпуклые функционалы возникают в самых разных ситуациях \cite{Lu_Nesterov_2018}. Рассмотрение таких задач на ограниченных допустимых множествах естественно приводит к вариационным неравенствам с соответствующим предположением относительной сильной монотонности оператора.

\begin{example}{Параллелепипедно-симплексные игры.} \label{ex:box}
Рассмотрим задачу централизованного решения вариационного неравенства, связанного с параллелепипедно-симплексными играми (box-simplex games) \cite{Sid} следующего вида:
\begin{equation}\label{saddle_box_simplex}
    \min_{y \in [-1, 1]^n} \max_{z \in \triangle_{n}} f(y, z),\\
    f(y, z) := z^\top A y - \langle b, z\rangle + \langle c, y\rangle,
\end{equation}
где $A \in \mathbb{R}^{n \times n}$~--- положительно определённая матрица, а $ z \in \triangle_{n}$, где $\triangle_{n}$ ~--- стандартный единичный симплекс в пространстве $\mathbb{R}^{n}$. Параллелепипедно-симплексные игры обобщают задачу $\ell_\infty$-регрессии с параллелепипедными ограничениями, имеющую следующий вид:
\begin{equation*}
    \min_{y \in [-1, 1]^n} \|A y - b\|_\infty,
\end{equation*}
$\|x\|_{\infty} := \max_k|x_k|$, где $x = (x_1, x_2, \ldots, x_n) \in \mathbb{R}^{n}$.

Приведённая задача эквивалентна задаче \eqref{eq:1} решения вариационного неравенства, если выбрать для $x := (y, z)^\top$
\begin{equation}\label{eq:box-simplex1}
    \hat{g}(x) := (A^\top z + c, \; b - A y )^\top.
\end{equation}
Оператор  $\hat{g}$ монотонен, поскольку исходная седловая задача \eqref{saddle_box_simplex} является выпуклой по $y$ (при любом $z$) и вогнутой по $z$ (при любом $y$).

Рассмотрим оператор $g$, его относительно сильно монотонное приближение вида
\begin{equation}\label{eq:box-simplex}
\begin{split}
    g(x) := (A^Tz + c + \mu_y \nabla_y d(y,z), \\
     b - Ay + \mu_z\nabla_z d(y,z))^\top,
\end{split}
\end{equation}
где $\mu_y, \mu_z >0$~--- некоторые фиксированные параметры с достаточно малыми положительными значениями, а согласно
\cite{Sid} прокс-функция в \eqref{Brg_form} может быть выбрана следующим образом:
\begin{align*}
    d(x) &= z^\top |A| (y)^2 + 10 \|A\|_\infty \sum_{i=1}^n z_i \ln z_i,
\end{align*}
где $|\cdot|$ и $(\cdot)^2$ действуют покомпонентно, $\|A\|_{\infty} := \sup_{\|x\|_\infty = 1} \|A x\|_\infty$.

Ясно, что введённый в \eqref{eq:box-simplex} оператор в полученной задаче решения вариационного неравенства~--- $\min\{\mu_y, \mu_z\}$-сильно монотонный относительно выбранной указанным выше способом прокс-функции $d$. Поэтому к этой задаче можно непосредственно применить алгоритм~\ref{alg1}, что и будет сделано далее в разделе экспериментов.
\end{example}

В качестве ещё одного примера рассмотрим задачу централизованной распределённой минимизации эмпирического риска в предположении схожести слагаемых \cite{Hendr}.

\begin{example}{Минимизация эмпирического риска.}\label{min_risk}\\
Рассмотрим задачу централизованной распределённой минимизации эмпирического риска в предположении схожести слагаемых \cite{Hendr},

\begin{align}\label{EmpirProbl}
\nonumber F(x) &:= \frac{1}{m} \sum_{j=1}^{m} f_{j}(x) = \frac{1}{n m} \sum_{j=1}^{m} \sum_{i=1}^{n} \ell(x, z_{i}^{(j)})\\
&=\frac{1}{N}\sum\limits_{i=1}^N \ell(x,z_i) \longrightarrow\min\limits_{x\in Q}
\end{align}
в предположении, что исходные данные представляют из себя набор из $m$ выборок размера $n$, каждая из которых хранится на одном из $m$ серверов. При этом для достаточно большого $n$ все $f_j$ есть $\mu$-сильно выпуклые и $L$-гладкие (удовлетворяют условию Липшица градиента с константой $L > 0$) функционалы, которые можно считать статистически схожими \cite{Hendr}. Такая схожесть может быть описана в виде предположения \cite{Hendr} о том, что для всякого $x$
$$
\|\nabla^2 F(x) - \nabla^2 f_j (x)\|_2 \leq \gamma
$$
при всяком $j$ для некоторого достаточно малого $\gamma >0$. Здесь и далее $\|A\|_2 := \max\limits_{\|x\|_2 \leq 1}\|Ax\|_2 $ для евклидовой нормы $\|\cdot\|_2$. При этом предполагается, что существует центральный сервер (ему соответствует функционал $\overline{f}$), на который передаётся информация о градиентах $f_j$ в текущей точке, но не передаётся информация о значениях $f_j$. В \cite{Hendr} показано, что при таком допущении можно ввести прокс-функцию
\begin{equation}\label{prox_risk}
d(x):= \overline{f}(x) + \frac{\gamma}{2}\|x\|_2^2,
\end{equation}
и для соответствующей дивергенции Брэгмана $V(y, x) = d(y) - d(x) - \langle \nabla d(x), y -x \rangle$ функция $F$ будет относительно $1$-гладкой и относительно $\frac{\mu}{\mu + 2\gamma}$-сильно выпуклой, т.~е. для всяких $x, y \in Q$ верны неравенства
\begin{gather*}
    F(y) \leq F(x) + \langle \nabla F(x), y - x\rangle + V(y, x),\\
    F(y) \geq F(x) + \langle \nabla F(x), y - x\rangle + \frac{\mu}{\mu + 2\gamma} V(y, x).
\end{gather*}
	
Это означает, что при $\gamma \ll L$ для задач минимизации эмпирического риска (целевой функционал $L$-гладкий и $\mu$-сильно выпуклый) можно улучшить оценку $$O\left(\frac{L}{\mu}\log\frac{1}{\varepsilon}\right)$$
сложности неускоренного градиентного метода до $$O\left(\left(1+\frac{\gamma}{\mu}\right)\log\frac{1}{\varepsilon}\right),$$
используя, по-прежнему, только информацию первого порядка.
	
Отметим, что при сопоставимых значениях параметров $\gamma$ и $\mu$ (этого ввиду малости $\gamma$ возможно добиться, к примеру, регуляризацией задачи) такая оценка близка к следующей оценке сложности ускоренных методов:
$$O\left(\sqrt{1+\frac{\gamma}{\mu}}\log\frac{1}{\varepsilon}\right).$$
При этом на каждой итерации метода, выполняемой центральным узлом, доступна информация о градиенте целевого функционала $F$, но не доступна информация о значении функционала $F$. Доступность информации о градиенте позволяет рассматривать поставленную задачу минимизации эмпирического риска как задачу отыскания решения вариационного неравенства с относительно гладким и относительно сильно монотонным оператором $g = \nabla F$. Такой подход позволяет, в частности, предложить метод с полной адаптивной настройкой на параметр гладкости задачи \eqref{EmpirProbl}, чего не удалось добиться для метода в \cite{Hendr}.
\end{example}

Статья состоит из введения, заключения и трёх основных разделов. В разделе 2 рассматривается адаптивный  алгоритм \ref{alg1} для вариационных неравенств  с относительно гладким и относительно сильно монотонным оператором. Показано, что увеличение количества отрешиваний вспомогательных подзадач, вызванное адаптивностью, некритично влияет на количество итераций. Раздел 3 содержит обобщение результатов раздела 2 на случай, когда оператор вариационного неравенства удовлетворяет условию $\delta$-обобщенной гладкости \cite{Stonyakin}. В разделе 4 приведены численные эксперименты, демонстрирующие работоспособность адаптивного подбора параметров для алгоритма \ref{alg1}.

\He{АДАПТИВНЫЙ МЕТОД ДЛЯ ВАРИАЦИОННЫХ НЕРАВЕНСТВ С ОТНОСИТЕЛЬНО ГЛАДКИМ СИЛЬНО МОНОТОННЫМ ОПЕРАТОРОМ}\label{sec:2}

В этом разделе, отталкиваясь от недавно предложенного в работе \cite{Sid} метода, мы рассмотрим анонсированный в \cite{Titov} алгоритм~\ref{alg1} без использования техники рестартов в случае относительно гладкого сильно монотонного оператора. Напомним, что  оператор $g$ называется $L$-гладким, если при любых $x, y, z \in Q$ справедливо следующее неравенство
\begin{equation}\label{rel_smooth}
\langle g(y)-g(z),x-z\rangle \leq LV(x,z) + LV(z,y),
\end{equation}
где $V(\cdot, \cdot)$~--- это дивергенция Брэгмана.

\begin{algorithm}[ht!]
\caption{Адаптивный метод первого порядка для вариационных неравенств с относительно сильно монотонными операторами.}
\label{alg1}
\begin{algorithmic}[1]
\REQUIRE $ x_0 \in Q, L_0 > 0, \mu >0, d(\cdot), V(\cdot, \cdot).$
\STATE $z_0 = \argmin_{u\in Q} d(u).$
\FOR{$k \geq  0 $}
\STATE Найти наименьшее целое $i_k \geq  0$ такое, что
\begin{equation}\label{eq_alg_1}
\begin{gathered}
\langle g(z_k) - g(w_k), z_{k+1} - w_k \rangle \leq \\
\leq L_{k+1} \left (V(w_k, z_k) + V(z_{k+1}, w_k) \right ),
\end{gathered}
\end{equation}
где $L_{k+1} = 2^{i_k - 1} L_k$, и
\begin{equation}\label{wk_alg1}
w_k = \argmin\limits_{y\in Q} \left \{ \left \langle \frac{1}{L_{k+1}} g(z_k), y \right \rangle  + V(y, z_k)\right \},
\end{equation}
\begin{equation}\label{zk1_alg1}
\begin{gathered}
z_{k+1} = \argmin\limits_{z \in Q} \bigg \{ \left \langle \frac{1}{L_{k+1}} g(w_k), z \right \rangle  +\\
+ V(z, z_k) + \frac{\mu}{L_{k+1}} V(z, w_k) \bigg \}.
\end{gathered}
\end{equation}
\ENDFOR
\ENSURE $z_k$.
\end{algorithmic}
\end{algorithm}
\begin{theorem}
Пусть $g$~--- $\mu$-относительно сильно монотонный оператор и $z_*$~--- точное решение вариационного неравенства \eqref{eq:1}. Тогда для алгоритма \ref{alg1} верны следующие неравенства:
\begin{equation}\label{estim_alg3}
V(z_*, z_{k+1}) \leq \prod_{i=0}^k\left ( 1 + \frac{\mu}{L_{i+1}} \right )^{-1}  V(z_*, z_0).
\end{equation}
Если, кроме этого, $L_0 \leq 2L$, то имеет место неравенство
\begin{equation}\label{eq1}
V(z_*, z_{k+1}) \leq \left ( 1 + \frac{\mu}{2 L} \right )^{-(k+1)} V(z_*, z_0).
\end{equation}
\end{theorem}
\begin{proof}
Для $w_k$ и $z_{k+1}, \; \forall k \geq 0$, в \eqref{wk_alg1} и \eqref{zk1_alg1} имеем
\begin{equation}\label{e1}
\begin{gathered}
\frac{1}{L_{k+1}} \langle g(z_k), w_k - z_{k+1}\rangle \leq V(z_{k+1}, z_k) - \\
- V(z_{k+1}, w_k)- V(w_k, z_k)
\end{gathered}
\end{equation}
и
\begin{equation}\label{e2}
\begin{gathered}
\frac{1}{L_{k+1}} \langle g(w_k), z_{k+1} - z_* \rangle \leq  V(z_*, z_k) -\\
- V(z_*, z_{k+1}) - V(z_{k+1}, z_k) + \\
+ \frac{\mu}{L_{k+1}} (V(z_*, w_k) - V(z_*, z_{k+1})).
\end{gathered}
\end{equation}
После сложения \eqref{e1} и \eqref{e2} получим
\begin{equation*}
    \begin{aligned}
        & \frac{1}{L_{k+1}} \Big(\langle g(z_k), w_k - z_{k+1}\rangle + \langle g(w_k), z_{k+1} - z_* \rangle \Big)
        \\& \leq - ( V(w_k, z_k)+  V(z_{k+1}, w_k) ) + V(z_*, z_k) -
        \\& - V(z_*, z_{k+1}) + \frac{\mu}{L_{k+1}} (V(z_*, w_k) - V(z_*, z_{k+1})).
    \end{aligned}
\end{equation*}
Используя неравенство \eqref{eq_alg_1}, имеем
\begin{equation*}
\begin{gathered}
\langle g(z_k), z_{k+1} - z_k \rangle \leq \langle g(w_k), z_{k+1} - w_k \rangle +\\
+\langle g(z_k), w_k - z_k \rangle + L_{k+1} \big ( V(w_k, z_k) + \\
+ V(z_{k+1}, w_k) \big ),
\end{gathered}
\end{equation*}
следовательно
\begin{equation*}
\begin{gathered}
\frac{1}{L_{k+1}} \Big( \langle g(w_k), w_k - z_* \rangle - \mu V(z_*, w_k)\Big ) \leq \\
\leq V(z_*, z_k) - \left ( 1 + \frac{\mu}{L_{k+1}} \right ) V(z_*, z_{k+1}).
\end{gathered}
\end{equation*}
Таким образом,
$$
    V(z_*, z_{k}) \geq  \left ( 1 + \frac{\mu}{L_{k+1}} \right ) V(z_*, z_{k+1}),
$$
или
$$
    V(z_*, z_{k+1}) \leq \left ( 1 + \frac{\mu}{L_{k+1}} \right )^{-1} V(z_*, z_{k}).
$$

Далее, по рекурсии
$$V(z_*, z_{k+1}) \leq \prod_{i=0}^k\left ( 1 + \frac{\mu}{L_{i+1}} \right )^{-1}  V(z_*, z_0).$$
Учитывая, что $L_i \leq 2L,\;\forall i$, получаем
$$V(z_*, z_{k+1}) \leq \left ( 1 + \frac{\mu}{2 L} \right )^{-(k+1)} V(z_*, z_0).$$
\end{proof}

\begin{remark}\label{rem1}
Предложенный подход позволяет настраиваться на подходящие локальные параметры $L_{k+1}$ ($k\geq 0$) на итерациях алгоритма~\ref{alg1}. Однако при этом возможно увеличение числа обращений к вспомогательной подпрограмме п.~3 листинга. Тем не менее можно показать, что такое увеличение некритично. Действительно, ясно, что при условии $L_0 \leq 2L$ имеем $L_{k+1}\leq 2L$.  Тогда $L_{k+1}=2^{i_k - 2}L_k$ (где $i_k$~--- количество вспомогательных шагов п.~3 листинга алгоритма~\ref{alg1} на $k$-й итерации). Далее, мы получаем $\frac{L_{k+1}}{L_k}=2^{i_k - 2}$, откуда
\begin{gather}
\frac{L_N}{L_0}=\frac{L_N}{L_{N-1}} \frac{L_{N-1}}{L_{N-2}} \dotsi \frac{L_1}{L_0}=2^{i_0 + i_1 + \ldots + i_{N-1}-2N}.
\end{gather}
С другой стороны,
\begin{gather}
\frac{2L}{L_0}\geq\frac{L_N}{L_0}=2^{i_0 + i_1 + \ldots + i_{N-1}-2N},
\end{gather}
откуда получаем
\begin{gather}
i_0 + i_1 + \ldots + i_{N-1}\leq 2N + \log_2 \frac{2L}{L_0}.
\end{gather}
Это означает, что среднее количество обращений к п.~3 листинга алгоритма~\ref{alg1} сопоставимо с общим количеством итераций $N$.
Если вместо $L_0 \leq 2L$ можно допустить лишь $L_0 \leq C L$ для некоторой $C > 0$, то это приведёт лишь к изменению констант в рассуждениях выше, а принципиальный вывод сохранится.
\end{remark}

Для примера~\ref{min_risk} замечание~\ref{rem1} означает, что адаптивность не может привести к резкому росту числа коммуникаций центрального узла сети с периферийными.

\He{ОБОБЩЕНИЕ РЕЗУЛЬТАТА НА ЗАДАЧИ С $\delta$-ОБОБЩЁННЫМ УСЛОВИЕМ ГЛАДКОСТИ ОПЕРАТОРА ВАРИАЦИОННОГО НЕРАВЕНСТВА}\label{sec:3}

В этом разделе мы детально обоснуем анонсированные в \cite{Titov} теоретические результаты для двух вариантов алгоритма \ref{alg1} в случае, когда оператор $g$ удовлетворяет следующему условию обобщённой гладкости:
\begin{equation}\label{eq3}
    \langle g(y)-g(z), x-z\rangle \leq LV(x,z) + LV(z,y) +\delta
    \end{equation}
для некоторого фиксированного $\delta > 0$ при произвольных $x, y, z \in Q$. Параметр $\delta$ указывает на возможность применения рассматриваемых подходов на классе задач с обобщённым свойством относительной гладкости для вариационного неравенства. Например, такое обобщение позволяет применять рассматриваемую в этой статье методику к вариационным неравенствам с некоторыми типами нелипшицевых операторов и седловым задачи с некоторыми типами или обобщённо гладких функционалов Более подробно мотивирующие примеры приведены в статье \cite{Stonyakin} и докладе на конференции \cite{Titov}.

Рассмотрим следующий алгоритм \ref{alg2} и докажем для него следующий результат, анонсированный в \cite{Titov}.

\begin{algorithm}[htb]
\caption{Адаптивный метод для вариационных неравенств с $\delta$-обобщённо гладкими и относительно сильно монотонными операторами.}
\label{alg2}
\begin{algorithmic}[1]
\REQUIRE $x_0 \in Q,\;\delta >0,\;L_0 > 0,\;\mu >0,\;d(\cdot),\;V(\cdot, \cdot).$
\STATE $z_0 = \argmin_{u\in Q} d(u).$
\FOR{$k \geqslant  0$}
\STATE Найти наименьшее целое $i_k \geq 0$ такое, что
\begin{equation}\label{eq_alg_2}
\begin{gathered}
\langle g(z_k) - g(w_k), z_{k+1} - w_k \rangle \leq
\\
\leq L_{k+1} \left (V(w_k, z_k) + V(z_{k+1}, w_k) \right ) + \delta,
\end{gathered}
\end{equation}
где $L_{k+1} = 2^{i_k - 1} L_k$, и
\begin{equation}\label{wk_alg2}
    w_k = \argmin\limits_{y \in Q} \left \{ \left \langle \frac{1}{L_{k+1}} g(z_k), y \right \rangle  + V(y, z_k)\right \},
\end{equation}
\begin{equation}\label{zk1_alg2}
\begin{gathered}
    z_{k+1} = \argmin\limits_{z \in Q} \bigg \{ \left \langle \frac{1}{L_{k+1}} g(w_k), z \right \rangle + V(z, z_k) \\
    + \frac{\mu}{L_{k+1}} V(z, w_k) \bigg \}.
    \end{gathered}
\end{equation}
\ENDFOR
\ENSURE $z_k$.
\end{algorithmic}
\end{algorithm}

\begin{theorem}\label{th4}
Пусть $g$~--- $\mu$-относительно сильно монотонный оператор и $z_*$~--- точное решение вариационного неравенства \eqref{eq:1}. Тогда для алгоритма~\ref{alg2} справедливо следующее неравенство:
\begin{equation}\label{estim_alg4}
\begin{gathered}
V(z_*, z_{k+1}) \leq  \prod_{i = 1}^{k+1} \left ( 1 + \frac{\mu}{L_{i}} \right )^{-1}  V(z_*, z_0) +\\
+  \frac{\delta}{L_{k+1} + \mu} + \sum_{j = 1}^k \frac{\delta}{L_j + \mu} \prod_{i = j+1}^{k+1} \left ( 1 + \frac{\mu}{L_i} \right )^{-1}.
\end{gathered}
\end{equation}
\end{theorem}

\begin{proof}
Для $w_k$ и $z_{k+1}$ при всяких $ k \geq 0$ из \eqref{wk_alg2} и \eqref{zk1_alg2} получаем, что
\begin{equation}\label{i1}
\begin{gathered}
\frac{1}{L_{k+1}} \langle g(z_k), w_k - z_{k+1}\rangle \leq V(z_{k+1}, z_k) \\
- V(z_{k+1}, w_k)- V(w_k, z_k)
\end{gathered}
\end{equation}
и
\begin{equation}\label{i2}
\begin{gathered}
\frac{1}{L_{k+1}} \langle g(w_k), z_{k+1} - z_* \rangle \leq V(z_*, z_k) - V(z_*, z_{k+1}) - \\
- V(z_{k+1}, z_k) +\frac{\mu}{L_{k+1}} (V(z_*, w_k) - V(z_*, z_{k+1})).
\end{gathered}
\end{equation}
После сложения \eqref{i1} и \eqref{i2} получим
\begin{equation*}
\begin{aligned}
 &  \frac{1}{L_{k+1}}  \Big(\langle g(z_k), w_k - z_{k+1}\rangle + \langle g(w_k), z_{k+1} - z_* \rangle \Big) \leq
 \\& \leq  - ( V(w_k, z_k)+  V(z_{k+1}, w_k) ) + V(z_*, z_k)  -
 \\& -  V(z_*, z_{k+1}) +\frac{\mu}{L_{k+1}} (V(z_*, w_k) - V(z_*, z_{k+1})).
\end{aligned}
\end{equation*}
Используя неравенство \eqref{eq_alg_2}, получим
\begin{equation*}
\begin{aligned}
     \langle g(z_k),& z_{k+1} - z_k \rangle \leq \langle g(w_k), z_{k+1} - w_k \rangle +
    \\&
    + \langle g(z_k), w_k - z_k \rangle + L_{k+1} \big ( V(w_k, z_k)
    \\&
+ V(z_{k+1}, w_k) \big ) + \delta,
\end{aligned}
\end{equation*}
следовательно
\begin{equation*}
\begin{gathered}
\frac{1}{L_{k+1}} \Big( \langle g(w_k), w_k - z_* \rangle - \mu V(z_*, w_k) - \delta \Big ) \leq \\
\leq V(z_*, z_k) - \left ( 1 + \frac{\mu}{L_{k+1}} \right ) V(z_*, z_{k+1}).
\end{gathered}
\end{equation*}
Таким образом,
$$
    V(z_*, z_{k}) \geq  \left ( 1 + \frac{\mu}{L_{k+1}} \right ) V(z_*, z_{k+1}) - \frac{\delta}{L_{k+1}},
$$
или
$$
    V(z_*, z_{k+1}) \leq \left ( 1 + \frac{\mu}{L_{k+1}} \right )^{-1} V(z_*, z_{k}) + \frac{\delta}{L_{k+1} + \mu}.
$$
Далее, применяя рекурсию по $k$ в последнем неравенстве, приходим к требуемому неравенству
$$
    V(z_*, z_{k+1}) \leq \left ( 1 + \frac{\mu}{L_{k+1}} \right )^{-1} V(z_*, z_{k}) + \frac{\delta}{L_{k+1} + \mu} \leq
$$
$$
    \leq \ldots \leq \prod_{i = 1}^{k+1} \left ( 1 + \frac{\mu}{L_{i}} \right )^{-1} V(z_*, z_0) + \frac{\delta}{L_{k+1} + \mu} +
$$
$$
  +  \sum_{j = 1}^k \frac{\delta}{L_j + \mu} \prod_{i = j+1}^{k+1} \left ( 1 + \frac{\mu}{L_i} \right )^{-1}.
$$
\end{proof}

Рассмотрим ещё одну вариацию алгоритма \ref{alg1} для вариационных неравенств с $\delta$-обобщенно гладкими операторами. За счёт модификации критерия выхода из итерации этот подход позволяет улучшить оценку качества выдаваемого методом решения в части влияния на него значения параметра $\delta>0$.

\begin{algorithm}[htb]
\caption{Вариант адаптивного метода для вариационных неравенств с $\delta$-обобщённо гладкими и относительно сильно монотонными операторами.}
\label{alg3}
\begin{algorithmic}[1]
\REQUIRE $x_0 \in Q,\;\delta >0,\;L_0 > 0,\;\mu >0,\;d(\cdot),\;V(\cdot, \cdot).$
\STATE $z_0 = \argmin_{u\in Q} d(u).$
\FOR{$k \geqslant  0$}
\STATE Найти наименьшее целое $i_k \geq 0$ такое, что
\begin{equation}\label{eq_alg_3}
\begin{gathered}
\langle g(z_k) - g(w_k), z_{k+1} - w_k \rangle \leq
\\
\leq L_{k+1} \left (V(w_k, z_k) + V(z_{k+1}, w_k) \right ) + L_{k+1}\delta,
\end{gathered}
\end{equation}
где $L_{k+1} = 2^{i_k - 1} L_k$, и
\begin{equation}\label{i1alg3}
    w_k = \argmin\limits_{y \in Q} \left \{ \left \langle \frac{1}{L_{k+1}} g(z_k), y \right \rangle  + V(y, z_k)\right \},
\end{equation}
\begin{equation}\label{i2alg3}
\begin{gathered}
    z_{k+1} = \argmin\limits_{z \in Q} \bigg \{ \left \langle \frac{1}{L_{k+1}} g(w_k), z \right \rangle + V(z, z_k) \\
    + \frac{\mu}{L_{k+1}} V(z, w_k) \bigg \}.
    \end{gathered}
\end{equation}
\ENDFOR
\ENSURE $z_k$.
\end{algorithmic}
\end{algorithm}

Для алгоритма \ref{alg3} докажем следующий результат, анонсированный в \cite{Titov}.
\begin{theorem}\label{th5}
Пусть $g$ $\mu$-относительно сильно монотонный оператор и $z_*$~--- точное решение вариационного неравенства \eqref{eq:1}. Тогда для алгоритма \ref{alg3} справедливы следующие неравенства:
\begin{equation}\label{estim_alg5}
\begin{gathered}
V(z_*, z_{k+1}) \leq \prod_{i=1}^{k+1} \left ( 1 + \frac{\mu}{L_{i}} \right )^{-1} V(z_*, z_0) +
\\
+\delta \left (1 + \sum_{j=1}^k \prod_{i=j+1}^{k+1} \left ( 1 + \frac{\mu}{L_i} \right )^{-1} \right ).
\end{gathered}
\end{equation}
Если, кроме этого, $L_0 \leq 2L$, то имеет место неравенство
\begin{equation}\label{ineq1}
\begin{gathered}
V(z_*, z_{k+1}) \leq \left ( 1 + \frac{\mu}{2L} \right )^{-(k+1)} V(z_*, z_0) + \\
+ \left ( 1 + \frac{2L}{\mu} \right ) \delta.
\end{gathered}
\end{equation}
\end{theorem}
\begin{proof}
Для $w_k$ и $z_{k+1}$ при всяких $ k \geq 0$ из \eqref{i1alg3} и \eqref{i2alg3} получаем, что
\begin{equation}\label{ooo}
\begin{gathered}
\frac{1}{L_{k+1}} \langle g(z_k), w_k - z_{k+1}\rangle \leq V(z_{k+1}, z_k) - \\
- V(z_{k+1}, w_k)- V(w_k, z_k),
\end{gathered}
\end{equation}
и
\begin{equation}\label{ppp}
\begin{gathered}
\frac{1}{L_{k+1}} \langle g(w_k), z_{k+1} - z_* \rangle \leq  V(z_*, z_k) - V(z_*, z_{k+1}) \\
- V(z_{k+1}, z_k)  + \frac{\mu}{L_{k+1}} (V(z_*, w_k) - V(z_*, z_{k+1})).
\end{gathered}
\end{equation}

После сложения \eqref{ooo} и \eqref{ppp} получим
\begin{equation*}
\begin{gathered}
\frac{1}{L_{k+1}}  \Big(\langle g(z_k), w_k - z_{k+1}\rangle +  \langle g(w_k), z_{k+1} - z_* \rangle \Big) \leq
\\
- ( V(w_k, z_k)+  V(z_{k+1}, w_k) ) + V(z_*, z_k) -
\\
- V(z_*, z_{k+1}) + \frac{\mu}{L_{k+1}}(V(z_*, w_k) - V(z_*, z_{k+1})).
\end{gathered}
\end{equation*}

Используя неравенство \eqref{eq_alg_3}, получим
$$
    V(z_*, z_{k+1}) \leq \left ( 1 + \frac{\mu}{L_{k+1}} \right )^{-1} V(z_*, z_{k}) + \frac{L_{k+1 }\delta}{L_{k+1} + \mu},
$$
или
$$
    V(z_*, z_{k+1}) \leq \left ( 1 + \frac{\mu}{L_{k+1}} \right )^{-1} V(z_*, z_{k}) + \delta,
$$
откуда
$$ V(z_*, z_{k+1}) \leq \prod_{i=1}^{k+1} \left ( 1 + \frac{\mu}{L_{i}} \right )^{-1} V(z_*, z_0) \, + $$

$$ + \, \delta \left (1 + \sum_{j=1}^k \prod_{i=j+1}^{k+1} \left ( 1 + \frac{\mu}{L_i} \right )^{-1} \right ).$$

Далее, если $L_0 \leq 2L$, то для всякого натурального $i$ верно $L_i \leq 2 L$, получаем
$$ V(z_*, z_{k+1}) \leq \left ( 1 + \frac{\mu}{2L} \right )^{-(k+1)} V(z_*, z_0) \, +$$
$$+ \, \left ( 1 + \frac{2L}{\mu} \right ) \delta,$$
что и требовалось.
\end{proof}

\He{ЧИСЛЕННЫЕ ЭКСПЕРИМЕНТЫ}\label{sec:exp}

В данном разделе для иллюстрации работоспособности предложенных выше алгоритмов приведем результаты вычислительных экспериментов по применению алгоритма~\ref{alg1} к различным задачам.

Программы для проведения экспериментов были реализованы на Python 3.4, все эксперименты проводились на компьютере с процессором Intel(R) Core(TM) i7-8550U CPU (1.80GHz, 4 ядра, 8 потоков). Оперативная память компьютера составляла 8 ГБ.

1) Первой была рассмотрена задача решения вариационного неравенства, связанная с параллелепипедно-симплексными играми и описанная в примере~\ref{ex:box}. Фигурирующие в постановке задачи данные были сгенерированы случайно следующим образом: $A = B B^\top$, $B_{ij}$ есть независимые случайные величины с равномерным распределением вероятностей на $[0, 0.001]$, $b,\;c \in \mathbb{R}^n$~--- с равномерным распределением вероятностей на $[0, 1]$.

К задаче размерности $n = 200$ с парами коэффициентов $\mu_z = \mu_y = 10^{-2}$ и $\mu_z = 10^{-2}$, $\mu_y = 10^{-6}$ были применены алгоритм~\ref{alg1}, соответствующий неадаптивный вариант \cite{Sid} и обычный экстраградиентный методы. Начальная точка выбиралась случайным образом, была одинаковой для всех запусков и имела компоненты с равномерным распределением вероятностей на $[0, 1]$.

\begin{figure}[ht!]
    \centering
    \includegraphics[width=7.5cm]{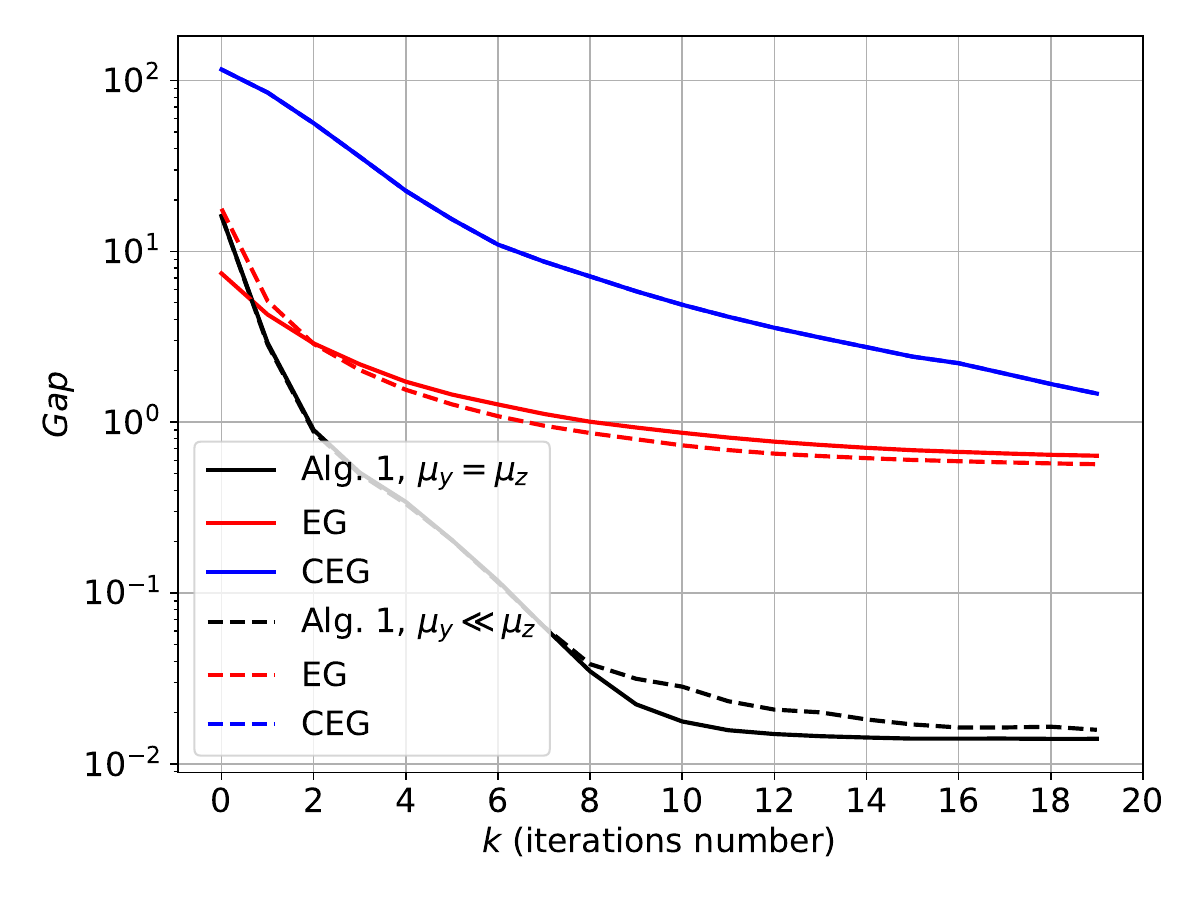}
    \caption{Результаты эксперимента по сравнению алгоритма~\ref{alg1} (Alg. 1) с соответствующим неадаптивным (EG) и классическим экстраградиентным (CEG) методами для задачи \eqref{eq:1}, \eqref{eq:box-simplex} с $\mu_y = \mu_z$ (сплошная линия) и $\mu_y \ll \mu_z$ (пунктир).}
    \label{Fig_box_simplex}
\end{figure}

Результаты сравнения алгоритмов показаны на рис.~\ref{Fig_box_simplex}. В качестве метрики качества для построения кривых сходимости используется зазор, определённый как
\begin{equation*}
    \operatorname{Gap}(y, z) := \max_{z' \in \triangle_n} f(y, z') - \min_{y' \in [-1, 1]^n} f(y', z)
\end{equation*}
и вычисляемый приблизительно с использованием реализации солвера SLSQP из библиотеки scipy. Из рис.~\ref{Fig_box_simplex} видно, 
что предлагаемый адаптивный алгоритм сходится заметно быстрее как своего неадаптивного варианта алгоритма \ref{alg1} с постоянным шагом, так и классического экстраградиентного метода.

2) Далее, рассмотрена распределённая централизованная задача гребневой регрессии \cite{acceleration_similarity}, имеющая структуру \eqref{EmpirProbl} для
\begin{equation}\label{eq:ridge}
    f_j(x) := \frac{1}{2s} \|A_j x - b_j\|_2^2 + \lambda \|x\|_2^2,\; j = 1, \ldots, m,
\end{equation}
для некоторого $\lambda > 0$. Таким образом, каждый агент $j$ владеет в качестве локальных данных матрицей $A_j \in \mathbb{R}^{s \times n}$ и вектором $b_j \in \mathbb{R}^s$. Данные были сгенерированы случайно в соответствии со следующими законами: $A_j \in \mathbb{R}^{s\times n}$ для каждого $j = 1, \ldots, m$ имеет компоненты с экспоненциальным распределением вероятностей с коэффициентом масштаба $1$ (в данном случае $\gamma > 1$) и со стандартным распределением Коши
(тогда параметр схожести слагаемых $\gamma = 10^{-2}$ в \eqref{eq:ridge}), а элементы $b_j \in \mathbb{R}^s$~--- с равномерным распределением вероятностей на $[0,1]$. В качестве множества $Q$ выберем евклидов $\ell_2$-шар в $\mathbb{R}^n$ с центром в $0 \in \mathbb{R}^n$ и радиусом $1$. Вслед за \cite{Hendr} выберем в качестве прокс-функции
\begin{equation*}
d(x) = f_1(x) + \frac{\gamma}{2}\|x\|_2^2,
\end{equation*}
где $\gamma$ оценивает сверху спектральную норму разности $\nabla^2 f_1(x) - \nabla^2 F(x)$ (её наибольшее собственное значение) \cite{Hendr}.

Алгоритм~\ref{alg1}, так же как соответствующий ему неадаптивный вариант с постоянным шагом и метод зеркального спуска \cite{Hendr}, который имеет следующий вид:
\begin{equation}\label{eq_MD_Hen}
x_{k+1} =  \argmin\limits_{x\in Q} \left \{ \left \langle  \nabla F(x_k), x \right \rangle  + V(x, x_k)\right \},
\end{equation}
в качестве начальной точки имел $x_0 = (1/\sqrt{n}, \dots, 1/\sqrt{n})^\top \in \mathbb{R}^n$, а коэффициенты в задаче были зафиксированы в следующих значениях: $n = 50,\;s = m = 100$.
\begin{figure}[ht!]
    \centering
    \includegraphics[width=7.5cm]{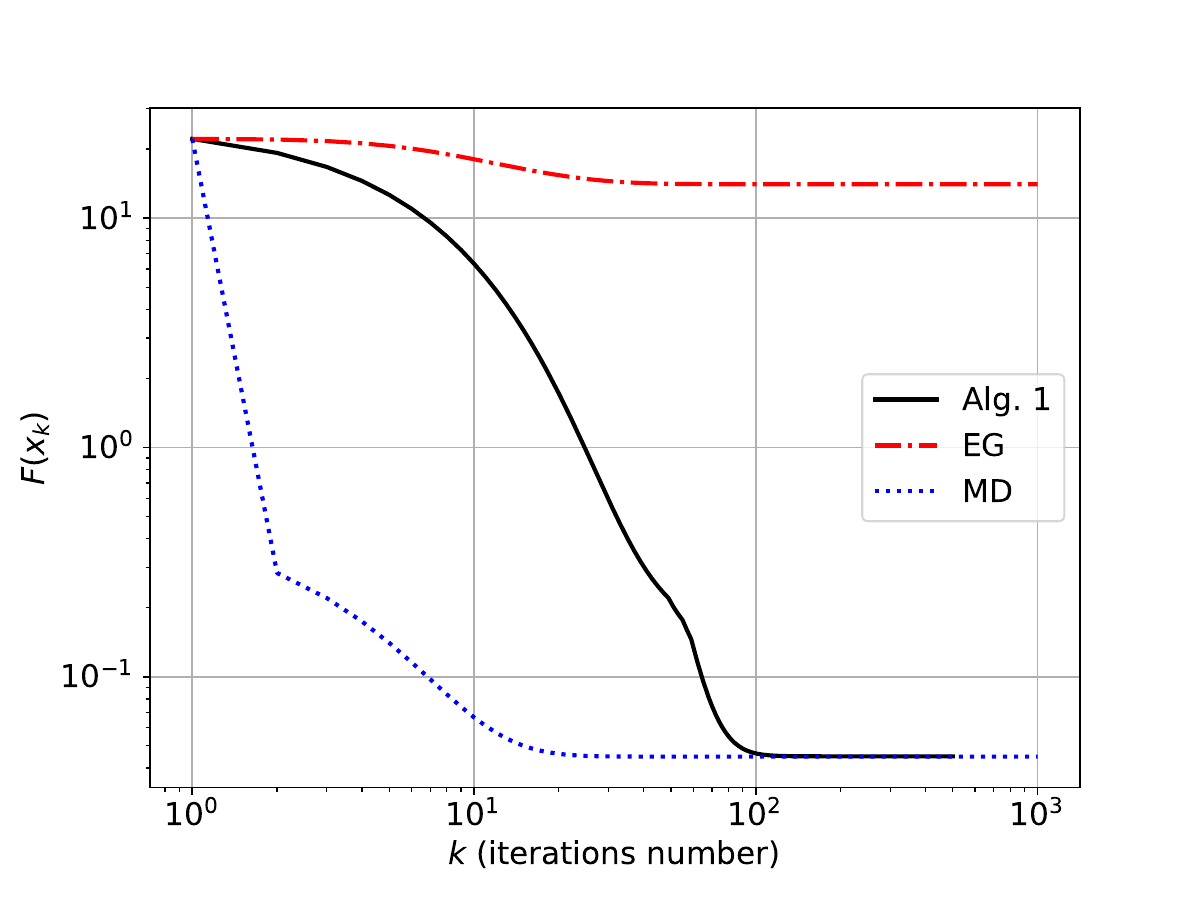}
    \includegraphics[width=7.5cm]{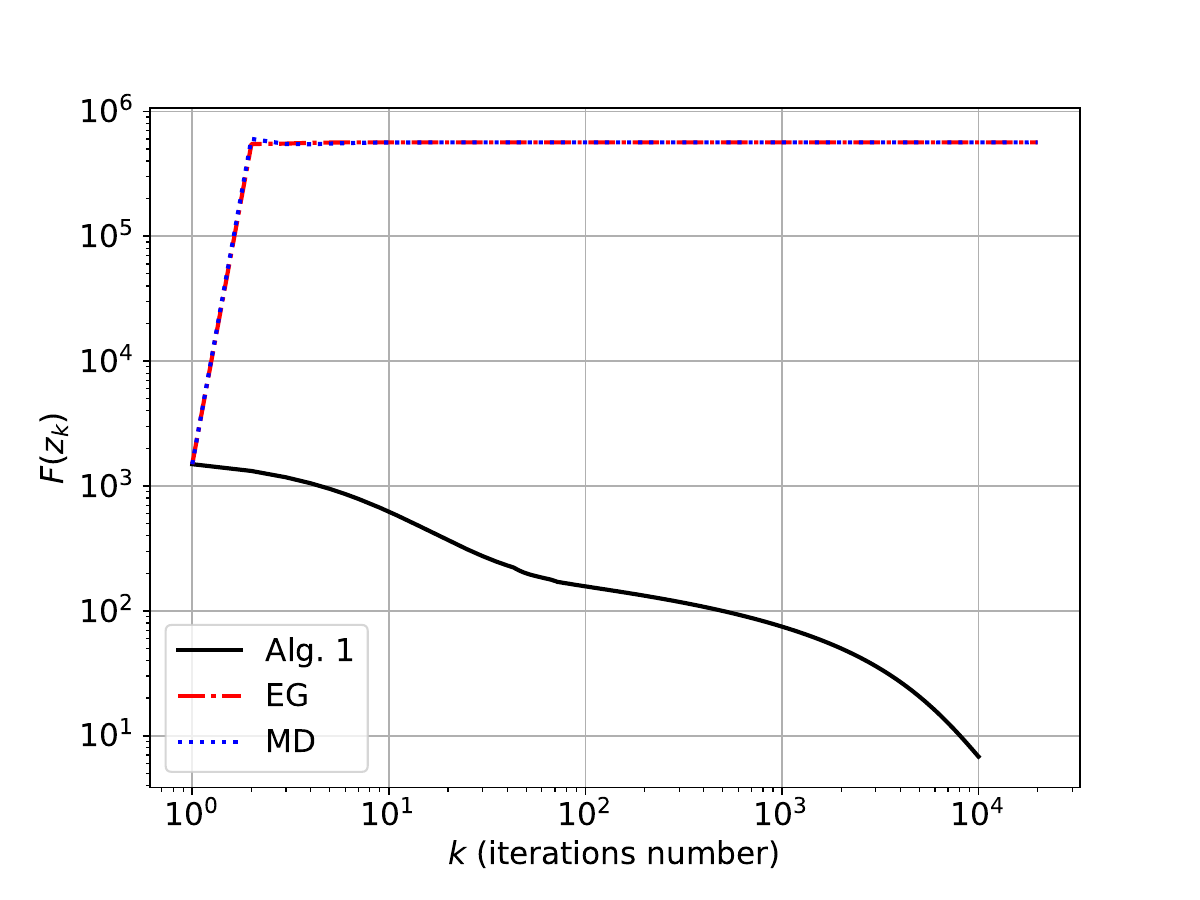}
    \caption{Результаты эксперимента по сравнению алгоритма~\ref{alg1} (Alg. 1) с соответствующим неадаптивным методом (EG) и методом зеркального спуска (MD) для задачи \eqref{EmpirProbl}, \eqref{eq:ridge} с $\lambda  = 10^{-1},\;\gamma > 1$ (сверху) и $\lambda  = 10^{-3},\;\gamma = 10^{-2} $ (снизу).}
    \label{Alg1_EG_MD_prob1}
\end{figure}

Результаты сравнения алгоритмов представлены на рис.~\ref{Alg1_EG_MD_prob1}. Построенные кривые сходимости отображают изменение значения целевой функции $F(x) = \frac{1}{m} \sum_{i=1}^m f_i(x)$ с увеличением числа итераций алгоритма~\ref{alg1}. Из рис.~\ref{Alg1_EG_MD_prob1} могут быть сделаны следующие наблюдения. Во-первых, среди подходов, основанных на сведении задачи распределённой оптимизации к задаче решения вариационного неравенства, адаптивный алгоритм \ref{alg1} сходится за существенно меньшее число итераций, чем его неадаптивный вариант. Во-вторых, оба соответствующих алгоритма уступают в эффективности подходу, решающему непосредственно исходную задачу оптимизации: можно видеть, что метод зеркального спуска наиболее эффективен среди подходов, применённых к рассмотренной задаче. Заметим также, что в наших экспериментах выбор параметра $\gamma$ меньшим, чем то предписано теоретической оценкой схожести функционалов, приводил к хорошей сходимости именно исследуемого нами алгоритма \ref{alg1} для вариационных неравенств. При этом оказалось, что наблюдалась расходимость соответствующего ему неадаптивного метода и метода зеркального спуска \eqref{eq_MD_Hen}. Представляется, что отмеченные факты интересно исследовать дополнительно. На данный момент основной задачей была демонстрация преимуществ адаптивного подбора шага в алгоритме \ref{alg1} по сравнению с известным ранее его неадаптивным вариантом \cite{Sid}.

\He{ЗАКЛЮЧЕНИЕ}

В работе исследован адаптивный метод (алгоритм~\ref{alg1}) решения вариационных неравенств с относительно гладкими и относительно сильно монотонными операторами. Получена оценка скорости сходимости рассматриваемого метода с использованием адаптивно подбираемых параметров. В частности, это даёт возможность применять подход и полученные результаты к достаточно популярному классу выпукло-вогнутых седловых задач вида ($Q_x$ и $Q_y$~--- выпуклые замкнутые множества)
$$\min_{x\in Q_x} \max_{y\in Q_y} f(x,y)$$
с соответствующими предположениями о гладкости функционала $f$:
$$\| \nabla_x f(x_1, y) - \nabla_x f(x_2, y) \|_2 \leq L_{xx} \|x_1 - x_2\|_2,$$
$$\| \nabla_x f(x, y_1) - \nabla_x f(x, y_2) \|_2 \leq L_{xy} \|y_1 - y_2\|_2,$$
$$\| \nabla_y f(x_1, y) - \nabla_y f(x_2, y) \|_2 \leq L_{xy} \|x_1 - x_2\|_2,$$
$$\| \nabla_y f(x, y_1) - \nabla_y f(x, y_2) \|_2 \leq L_{yy} \|y_1 -y_2\|_2,$$
$\forall x, x_1, x_2 \in Q_x$, $y, y_1, y_2 \in Q_y$ и норма $\|\cdot\|_2$ евклидова. Алгоритм~\ref{alg1} позволяет реализовать адаптивную настройку на все параметры гладкости $L_{xx}, L_{xy}, L_{yy} >0 $.

При этом показано (замечание~\ref{rem1}), что дополнительные процедуры адаптивного подбора параметров в исследованных выше алгоритмах не приводят к резкому росту сложности метода. Рассмотрено приложение к адаптивному методу для задач централизованной распределённой оптимизации (задача минимизации эмпирического риска из примера~\ref{min_risk}) со схожестью слагаемых.

Также исследована более общая ситуация, когда оператор вариационного неравенства удовлетворяет условию $\delta$-обобщённой гладкости \cite{Stonyakin} (в статье \cite{Stonyakin} рассматривался случай класса только монотонных операторов). В этой связи рассмотрен алгоритм~\ref{alg2}, получен соответствующий результат о сходимости и описано влияние параметра $\delta$ на качество выдаваемого решения. Проведены численные эксперименты по реализации алгоритма~\ref{alg1} для решения следующих задач: распределенная централизованная задача гребневой регрессии со схожестью функций-слагаемых и задача решения вариационного неравенства, связанного с параллелепипедно-симплексными играми. Представляется, что выполненные эксперименты продемонстрировали достаточно хорошую эффективность алгоритма \ref{alg1}.

\He{ФИНАНСИРОВАНИЕ} Работа над разделом \ref{sec:2} выполнена при частичной поддержке программы стратегического академического лидерства “Приоритет-2030” (соглашение №~075-02-2021-1316 от 30.09.2021). Работа над разделом \ref{sec:exp} выполнена при финансовой поддержке гранта Президента Российской Федерации для государственной поддержки ведущих научных школ НШ775.2022.1.1.

\label{lastpage}

\end{document}